\documentclass[12pt]{amsart}
\usepackage{amssymb, latexsym, amsmath, amscd, array, graphicx, enumerate, ,tikz-cd}

\usepackage{mathrsfs}%For mathscr

\swapnumbers
\numberwithin{equation}{section}

\def\m{\medskip}

\newtheorem{thm}{Theorem}[section]

\newtheorem{proposition}[thm]{Proposition}

\newtheorem{question}[thm]{Question}

\newtheorem{prop}[thm]{Proposition}

\theoremstyle{definition}
\newtheorem{definition}[thm]{Definition}

\newtheorem{df}[thm]{Definition}
\newtheorem{example}[thm]{Example}

\newcommand\defref{Definition~\ref}

\newcommand{\RR}{\mathbb{R}}
\newcommand{\ZZ}{\mathbb{Z}}

\newcommand{\ga}{\alpha}

\def\TC{\operatorname{TC}}
\def\rTC{\operatorname{rTC}}
\def\wTC{\operatorname{wTC}}
\def\sTC{\operatorname{sTC}}
\def\TCP{\operatorname{TC^{Pav}}}
\def\TCRD{\operatorname{TC^{RD}}}
\def\TCMW{\operatorname{TC^{MW}}}
\def\TCS{\operatorname{TC^{S}}}
\def\pt{\operatorname{pt}}
\def\pr{\operatorname{pr}}

\def\wsecat{\operatorname{wsecat}}
\def\cat{\operatorname{cat}}
\def\wcat{\operatorname{wcat}}
\def\qscat{\operatorname{qscat}}

\def\srelcat{\operatorname{srelcat}}
\def\relcat{\operatorname{relcat}}

\def\secat{\operatorname{secat}}
\def\sec{\operatorname{secat}}
\def\id{\operatorname{id}}
\def\zcl{\operatorname{zcl}}
\def\ts{\times}

\def\ov{\overline}

\long\def\forget#1\forgotten{} %

\begin{document}
\title[]{Relative LS categories and higher topological complexities of maps}

\author[Y. B. Rudyak]{Yuli B. Rudyak}
\address{Department of Mathematics, University of Florida, Gainesville, FL 32611-8105}
\email{rudyak@ufl.edu}

\author[S. Sarkar]{Soumen Sarkar}
\address{Department of Mathematical, Indian Institute of Technology Madras, Chennai 600036, India}
\email{soumen@iitm.ac.in}

\date{\today}
\thanks{
}

\maketitle
\abstract

In this paper, we study three relative LS categories of a map and study some of their properties. Then we introduce the `higher topological complexity' and `weak higher topological complexity' of a map. Each of them are homotopy invariants. We discuss some lower and upper bounds of these in invariants and compare them with previously known `topological complexities' of a map.  
\endabstract

\section{Introduction}
Let $p \colon E \to B$ be a fibration where $B$ is a path connected CW
space. The ``genus'' of  $p$  was introduced by Schwarz in \cite{sv}, and it is the minimum cardinality of the open coverings of $B$ such that on each open set in the covering there is a section of $p$. If no such integer exists, then by convention the genus of  $p $ is $ \infty$. James \cite{J} replaced the overworked term ``genus'' by ``sectional category'' which is denoted by $\secat(p)$. Now by convention people agree that $\secat p$ is one less than the Schwarz genus of $p$ in \cite{sv}. In particular, $\secat(p)=0$ if and only if $p$ has a section. We note that if $E$ is contractible and the map $p$ is surjective then $\secat(p)$ is equal to the classical Lusternik--Schnirelmann category (in short LS category throughout this paper) of  $B$. See \cite{CLOT} for  properties and applications  of the sectional category and the LS category.  

\vspace{0.1 cm}  

There is a generalization of LS category called the category of a map, \cite{BeGa}. We recall that the LS category of $f$ is the minimum number $\cat(f)$ such that $X$ can be covered by $\cat(f) + 1$ open subsets and the restriction of $f$ to each of these  open subsets is null-homotopic.  In particular, if $\id_X \colon X \to X$ is the identity map, then $\cat(\id_X) = \cat(X)$.  It is a homotopy invariant and
  \begin{equation}\label{eq_prod}
\cat(f \times g) \leq \cat(f) + \cat(g), \quad \cat(g\circ f)\leq \min\{\cat(g), \cat(f)\},
\end{equation} see \cite{CLOT, Sta}. The first inequality may be known as the product formula for the category of maps. 

\vspace{0.1 cm}  

Farber \cite{Far} introduced the topological complexity of a configuration space  to understand the navigational complexity of the motion of a robot. Interestingly, this invariant is also a particular case of the sectional category. We denote  the closed unit interval $[0,1]$ by $I$. Let $X$ be a path connected Hausdorff space and $X^{I}$ the free path space in $X$ equipped with the compact-open topology, and let
\[
\pi \colon X^I  \to X \times X
\]
be the free path fibration defined by $\pi(\alpha) = (\alpha(0), \alpha(1))$ for $\alpha\in X^I$. Then  $\secat(\pi)$ is equal to Farber's topological complexity, denoted by $\TC (X)$.

\m So, we have close relatives: the LS category and the topological complexity. Both are numerical invariants and are special cases of Schwarz genus (sectional category).
In view of the parallelism between LS category and TC, it seems reasonable to loop the presentation and introduce TC of mappings.

In this way A. Dranishnikov asked for an appropriate definition of the topological complexity of a map in his talk at the CIEM, Castro Urdiales in 2014. 
To answer this, several people developed a concept  of the ``(higher) topological complexity'' of a map. The basic goal in defining a map $f \colon X\to Y$ from one space (of states) $X$ to another space (of states) $Y$  is to study suitable properties of one from another. For example, if $f \colon X \to Y$ is a map and $f$  is nice enough then the complexity of $X$ could be approximated by the complexity of $Y$, and this expectation is very natural. On the other hand, topologists expect that if $X$ and $Y$ are not very complex, then the complexity of $f$ should also not be high. In particular if $X$ and $Y$ are contractible, then from the topological point of view, the complexity of a map from $X$ to $Y$ should be trivial. Furthermore, it seems reasonable to expect that topological complexity is a homotopy invariant.

 \vspace{0.1 cm}  We note that most of the definition of the topological complexity of a map lack basic above mentioned properties. For example, \cite {Pav, RaDe} violate these expectations. Our definition (see Section \ref{sec_top_comp_map} below) confirm the expectation, and so we believe that it is more appropriate according to a topological point of view.

\vspace{0.1 cm}  
Murillo and Wu~\cite{MuWu} gave a definition of the ``topological complexity'' of a (work) map which is a homotopy invariant. Our approach is different. 

 \vspace{0.1 cm}  
 The paper is organized as follows.
 In Section \ref{sec_rel_cat_weight}, we introduce the concept of relative LS category, quasi-strong LS category and strong relative category of a map which are denoted by $\relcat(f)$, $\qscat(f)$, and $\srelcat(f)$, respectively.  We study several properties of these invariants which generalize some of the properties of the category of a map. We show that the product formula holds for these new invariants, see Proposition \ref{prop_prod_relcat}, \ref{prop_prod_Acatf} and \ref{prop_prod_srcat}. We  give a cohomological lower bound of $\srelcat(f)$ in Theorem \ref{thm:cohom_low_bnd}.

\vspace{0.1 cm}  
In Section \ref{sec_top_comp_map}, we use the quasi-strong LS category to introduce the concept of ``higher topological complexity'', denoted by $\TC_n(f),\, n=2,3, \ldots, $,  of a map $f \colon X \to Y$. We show that it is a homotopy invariant and the growth of $\TC_n(f)$ is linear with respect to $n$. We compare our definition of $\TC_2(f)$ with the ``topological complexity'' of $f$ of \cite{Pav}, \cite{RaDe}, and \cite{MuWu}. 

 \m In Section \ref {sec_wtop_comp_map}, we recall the concept of weak sectional category and weak topological complexity following \cite{GaVa}. Then we introduce the concept of weak higher topological complexity for a map and study a few properties of this invariant. 

\m We note that a preliminary version of the paper appeared in Section 2 and 3 of \cite{RS1}.

\section{Some relative LS categories} \label{sec_rel_cat_weight}
 Many topological constructions exploit relative version, informally saying to use an object ``modulo subspaces''. In this section, we introduce three types of relative LS categories of a map and study some of its properties. In this paper, all pair $(X,A)$ of topological spaces are assumed to be CW pair and all maps are continuous.  

\begin{definition}[{\cite[Definition 7.1]{CLOT}}]\label{d:clot}
Let $(X, A)$ be a CW pair. The {\em relative category} $\cat(X, A)$ is the least non-negative integer $k$ such that $X$ can be covered by open sets $V_0, V_1, \ldots, V_k$ with $A\subseteq V_0$  and such that, for $i\geq 1$, the set $V_i$ is contractible in $X$, and there exists a homotopy of pairs $H \colon (V_0 \times I, A \times I) \to (X, A)$ with $H(-, 1)$ is given by the inclusion map $V_0 \hookrightarrow X$ and $H(V_0, 0) \subseteq A$. If no such $n$ exists, we say this is infinity.
\end{definition}

Extending the above definition, one can introduce the concept of the category of maps between the pairs of spaces.

\begin{definition}\label{d:relcat_map}
Let $f \colon (X, A)~\to~(Y, B)$ be a map of CW pairs.  The {\em relative category} of $f$, denoted by $\relcat(f)$, is the least non-negative integer $n$ such that $X$ can be covered by open sets $V_0, V_1, \ldots, V_n$ with $A \subseteq V_0$, and for $1\leq i \leq n$, the map $f|_{V_i}$ is inessential, and there exists a homotopy of pairs $H \colon (V_0 \times I, A \times I) \to (Y, B)$ with $H(-, 1)$ is the map $f|_{V_0}$ and $H(V_0, 0) \subseteq B$. If no such $n$ exists, we say this is infinity.
\end{definition}

Note that if $\id \colon (X, A) \to (X, A)$ is the identity map then $\relcat(\id) = \cat(X, A)$. Also from Definition \ref{d:clot} and \ref{d:relcat_map}, one gets the following.

\begin{prop}\label{prop_hom_inv_relcatf}
\begin{enumerate}[\rm (i)]
\item The number $ \relcat(f)$ is a homotopy invariant.
\item If $f \colon (X, A) \to (Y, B)$ and $g \colon (Y, B) \to (Z, C)$, then  $$\relcat(g \circ f) \leq \min \{\relcat(f), \relcat(g)\}.$$ In particular, $\relcat(f) \leq \min \{\cat(X, A), \cat(Y, B)\}$.
\end{enumerate}
\end{prop}

Next we extend the concept of categorical sequence of a space for maps.  

\begin{definition}\label{d:catseq_map}
Let $f \colon (X, A)~\to~(Y, B)$ be a map of CW pairs. The {\em relative categorical sequence of length n} for $f$ is a sequence of open sets $P_0, P_1, \ldots, P_n=X$ with $A \subseteq P_0$ such that $P_{i+1} - P_i$ is a subset of an open set $V_{i+1}$ on which  the map $f|_{V_{i+1}}$ is inessential for $1\leq i \leq n-1$, and there exists a homotopy of pairs $H \colon (P_0 \times I, A \times I) \to (Y, B)$ with $H(-, 1)$ is the map $f|_{P_0}$ and $H(P_0, 0) \subseteq B$.
\end{definition}
The proof of the following proposition is similar to that of \cite[Lemma 1.36]{CLOT} and generalizes the later.

\begin{prop}\label{prop_rel_catseq}
Let $f \colon (X, A)~\to~(Y, B)$ be a map of CW pairs such that $X$ is path-connected. Then $f$ has relative categorical sequence of length $n$ if and only if $\relcat(f) \leq n$.
\end{prop}

\begin{prop}\label{prop_prod_relcat}
Let $f \colon (X, A) \to (Y, B)$ and $g \colon (Z, C) \to (W, D)$ be two maps such that $X  \times Z$ is a normal space. Then, for the map  $$f \times g \colon (X \times Z, X \times C \cup A \times Z) \to (Y \times W, Y \times D \cup B \times W),$$
$\relcat(f \times g) \leq \relcat(f) + \relcat(g)$.
\end{prop}

\begin{proof}
Let   $\relcat(f)=k$ and $\relcat(g)=\ell$. Suppose that $X$ can be covered by open sets $U_0, U_1, \ldots, U_{k}$ with $A \subseteq U_0$,  and suppose that $Z$ can be covered by open sets $V_0, V_1, \ldots, V_{\ell}$ with $C \subseteq V_0$  such that the covering $\{V_i\}$ satisfies the  conditions in Definition \ref{d:relcat_map}. Let $P_i := \cup_{s=0}^i U_s $ and $Q_j := \cup_{t=0}^j V_j$ for $i=0, \ldots, k$ and $j=0, \ldots, \ell$. Then $P_i - P_{i-1} \subseteq U_i$ and $Q_j - Q_{j-1} \subseteq V_j$  for $i=1, \ldots, k$ and $j=1, \ldots, \ell$. Let $R_0 := X \cup V_0 \cup U_0 \cup Y$. So it is an open subset of $X \times Y$. Let $H \colon (U_0 \times [0,1], A \times [0, 1]) \to (Y, B)$ and $K \colon (V_0 \times [0,1], C \times [0, 1]) \to (W, D)$ be two homotopies of pairs such that $H(-, 1)$ is the map $f|_{U_0}$ with $H(U_0, 0) \subseteq B$, and  $K(-, 1)$ is the map $g|_{V_0}$ with $K(V_0, 0) \subseteq D$.  Put
\[
L(x, y, t) =(H(x, t), K(y, t))
\] 
for $(x, y) \in R_0$ and $t \in [0, 1]$. Then 
\[
L \colon (W_0 \times I, (X \times C \cup A \times Z)\times I) \to (Y \times W, Y \times D \cup B \times W)
\]
 is a homotopy of pairs with $L(-, 1)=f \times g|_{W_0}$ and $L(W_0,0)~\subseteq~B~\times~D$. 
 
 Define open subsets of $X \times Y$ by
\[
R_m:= \bigcup_{s =1}^m P_s \times Q_{m+1-s}
\]
for $m \geq 1$. Here $P_s = \emptyset$ if $s >k$ and $Q_r = \emptyset$ if $r >\ell$. Now using the arguments in  the proof of \cite[Theorem 1.37]{CLOT} and the restriction of the products of homotopies on $U_i \times V_j$ for some $i$ and $ j$, one can show that $R_0, R_1, \ldots, R_{m+\ell}$ is a relative categorical sequence of length $k+\ell$ for the map $f \times g$. Thus by Proposition \ref{prop_rel_catseq}, $\relcat(f \times g) \leq k+\ell$.    
\end{proof}
Note that if $A, B, C$ and $D$ are empty subsets then proposition \ref{prop_prod_relcat}  turns into \cite[Theorem 1.37]{CLOT}.

\begin{definition}\label{d:A-cat_map}
Given two CW pairs $(B, C)$ and $(X, A)$, consider a map $f \colon (B, C)~\to~(X, A)$. Define the {\em quasi-strong} LS category of the map $f$, denoted by $\qscat(f)$, is the least integer $n$ such that $B$ can be covered by open sets $V_0, V_1, \ldots, V_n$ such that  there is a homotopy $H_i \colon V_i \times [0,1] \to X$ such that $H_i{(-, 1)}$ is the map $f|_{V_i}$ and $H_i(V_i, 0) \subseteq A$ for $i=0, \ldots, n$. If no such $n$ exists, we say this is infinity.
\end{definition}
Note that if $\id \colon (X, A) \to (X, A)$ is the identity map then $\qscat(\id)$ is the Clapp-Puppe variation \cite{ClPu} of LS category. We may denote it by $\qscat(X, A):=\qscat(\id)$. From Definition \ref{d:A-cat_map}, one gets the following.

\begin{prop}\label{prop_inv_qscatf}
\begin{enumerate}[\rm (i)]
\item The number $\qscat(f)$ is a homotopy invariant.
\item   If $f \colon (B, C) \to (X, A)$ and $g \colon (X, A) \to (Y, D)$, then  $$\qscat(g \circ f) \leq \min \{\qscat(f), \qscat(g)\}.$$
\end{enumerate}
\end{prop}

\begin{prop}\label{prop:ineqality1}
  $\cat(X, A) \leq {\rm cat}(X - A) + 1$ for any CW pair $(X,A)$.
\end{prop}
\begin{proof}
Let $V_0$ be an open subset of $X$ containing $A$. Assume that there is a deformation retract from $V_0$ onto $A$. Such $V_0$ exists, since $(X, A)$ is a CW-pair. Together with this any categorical covering of $X - A$ satisfies Definition \ref{d:clot}, since $X -A$ is open in $X$.
\end{proof}

\begin{prop}\label{prop_prod_Acatf}
Let $f_i \colon (B_i, C_i) \to (X_i, A_i)$ be a map for $i=1, 2$ such that $B_1 \times B_2$ is a normal space. Then, for the map  \[
f_1 \times f_2 \colon (B_1 \times B_2, C_1 \times C_2) \to (X_1 \times X_1, A_1 \times A_2),
\]
we have $\qscat(f_1 \times f_2) \leq \qscat(f_1) + \qscat(f_2)$.
\end{prop}

\begin{proof}
The proof is similar to the proof of Proposition \ref{prop_prod_relcat}.    
\end{proof}

More strongly, one can define the following.
\begin{definition}\label{d:srelcat}
Let $(X, A) $ be a CW pair.
\begin{enumerate}
\item An open subset $U$ of $X$ is called \emph{strongly relative categorical} if there is a homotopy $H \colon U \times I \to X$ such that $H{(-, 1)}$ is the inclusion $U \hookrightarrow X$, $H(U, 0) \subseteq A$, and $H(a, t) = a $ for any $(a, t) \in (U \cap A) \times [0, 1]$.

\item We define {\em strong relative} LS category of a pair $(X, A)$,  denoted by $\srelcat(X,A)$, to be the least integer $n$ such that $X$ can be cover by $n+1$ strongly relative categorical sets. If no such $n$ exists, we say this is infinity.
\end{enumerate}
\end{definition}
In future we frequently omit the abbreviation LS and say about strong relative category. Note that $\cat (X)=\cat(X, x_0) = \srelcat (X, x_0)$ where $ X$ is path-connected.

\m Now we introduce the concept of ``strong relative category of a map''. Then we study their several properties in the remaining of this section.

\begin{definition}\label{def_rel_inessential}
Let $g \colon (B, C) \to (X, A)$ be a map of pairs.
We say that $g$ is called \emph{strongly relatively inessential} map if there is a homotopy $H \colon B \times I \to X$ such that  $H(-, 1)$ is $g$, $H(B, 0) \subseteq A$, and $H(c, t) =g(c) \in A$ for all $c \in C$ and $t \in I$.
\end{definition}
 
\begin{definition}\label{d:srelcat_map}
Let $f \colon (B, C) \to (X, A)$ be a map of pairs. Consider an open cover $\{B_0, \ldots, B_{\ell}\}$ of $B$ such that $f|_{(B_i, B_i\cap C)}$ is strongly relatively inessential for all $i$. The minimum $\ell$ with this property is called the  \emph {strong relative category} of $f$ and is denoted by $\srelcat(f)$. If no such $\ell$ exists, we put  $\srelcat(f)=\infty$ .
\end{definition}

Clearly, if $C= \emptyset$ or $C = \pt$ and $A = \pt$ then $\srelcat(f)$ turns into a known invariant $\cat(f)$ for $f \colon B \to X$, see  \cite[Exercise 1.16]{CLOT}. Also, $\qscat(f) \leq \srelcat(f)$. We remark that a different concept of relative category of a map is appeared in the paper \cite{DoEl}.

\begin{thm}\label{thm_hom_inv}
The number $ \srelcat(f)$ is a relative homotopy invariant.
\end{thm}
\begin{proof}
Let $f, g \colon (B, C) \to (X, A)$ be two maps that are homotopic relatively to $C$. Then there is a homotopy
$
H \colon B \times I \to X
$
 such that $H(b, 0) = f(b)$, $H(b, 1)=g(b)$ for all $b \in B$,
and $H(c, t) = f(c) = g(c) \in A$ for all $c \in C$ and $t \in I$. Let $B_i$ be an open subset of $B$ such that the map
\[
g|_{(B_i, B_i \cap C)} \colon (B_i, B_i\cap C) \to (X, A)
\]
 is strongly relatively inessential. Now we use the homotopy $H$ to conclude that  
 \[
 f|_{(B_i, B_i \cap C)} \colon (B_i, B_i\cap C) \to (X, A)
 \]
 is strongly relatively inessential. Therefore, $ \srelcat(f) \leq  \relcat(g)$. Similarly, the opposite inequality holds.
\end{proof}

Note that, for the identity map $\id \colon (X, A) \to (X, A)$, we have the equality $ \srelcat(\id) =  \srelcat(X, A)$. 
Moreover, we have the following.

\begin{proposition}\label{prop_inequality_1}
Let $f \colon (B, C) \to (X, A)$ and $g \colon (Y, D) \to (B, C)$ be two {\rm CW} maps. Then
\begin{enumerate}[\rm (i)]
\item $\srelcat(f) \leq \min\{\srelcat(B, C), ~~ \srelcat(X, A)\} $.
\item $\srelcat(f \circ g) \leq \min\{\srelcat(f), ~~ \srelcat(g)\} $.
\end{enumerate}
\end{proposition}

\begin{proof}
We give a proof for (i), and the arguments for other one are similar.
Let $B_i$ be an open subset of $B$ such that $B_i$ satisfies Definition \ref{d:srelcat}(1). So there is a  homotopy $H \colon B_i \times I \to B$ such that $H{(-, 1)}$ is the inclusion $B_i \hookrightarrow B$, $H(U, 0) \subseteq C$, and $H(c, t) =c$ for any $c \in B_i \cap C$ and $t \in I$. Then the homotopy $f \circ H$ satisfies Definition \ref{def_rel_inessential}. Similarly, one can show that $ \srelcat (f) \leq \srelcat(X, A) $.
\end{proof}

\begin{prop}\label{prop_prod_srcat}
Let $f_i \colon (B_i, C_i) \to (X_i, A_i)$ be a map for $i=1, 2$ such that $B_1 \times B_2$ is a normal space. Then, for the map  $$f_1 \times f_2 \colon (B_1 \times B_2, C_1 \times C_2) \to (X_1 \times X_1, A_1 \times A_2),$$
$\srelcat(f_1 \times f_2) \leq \srelcat(f_1) + \srelcat(f_2)$.
\end{prop}

\begin{proof}
The proof is similar to the proof of Proposition \ref{prop_prod_relcat}.    
\end{proof}
In particular $\srelcat(B \times X, C \times A) \leq \srelcat(B, C) + \srelcat(X, A).$

\begin{example}\label{ex:comp}
Let $D^n$ be the unit closed $n$-ball in $\mathbb{R}^n$ and $\partial{D^n}$ its boundary. Consider the identity map $\id \colon (D^n, \partial{D^n}) \to (D^n, \partial{D^n})$. We prove that $\relcat(\id)=1=\srelcat(\id)$ and $\qscat(\id) =0$.

\m First, note that $\cat(D^n, \partial{D^n}) \geq 1$, since the open subset $V_0=D^n$ of $D^n$ does not satisfy the homotopy condition in Definition \ref{d:clot}. Now, let $V_0= D^n - \{0\}$ be the punctured $n$-ball and $V_1 = D^n - \partial{D^n}$. Then $\{V_0, V_1\}$ is an open covering of $D^n$. Observe that this covering satisfy the conditions in Definition \ref{d:clot}. Therefore, $\cat(D^n, \partial{D^n})=1$.

Observe that there is a homotopy $H \colon D^n \times [0, 1] \to D^n$ satisfying the condition in \defref{d:A-cat_map}. Thus $\qscat(\id) =0$. 

Also, observe that $D^n$ cannot be strongly relative categorical for the pair $(D^n, \partial{D^n})$. However, $U_0 = D^n - (1, 0, \ldots, 0) $ and $U_1 = D^n - (-1, 0, \ldots, 0)$ are  strongly relative categorical for the pair $(D^n, \partial{D^n})$. Thus   
$\srelcat(\id)=1$.  
\qed
\end{example}

Example \ref{ex:comp} shows that the inequality in Proposition \ref{prop:ineqality1} can be strict. Also, note that Example \ref{ex:comp} implies that \defref{d:relcat_map} and \defref{d:A-cat_map} as well as  \defref{d:A-cat_map} and \defref{d:srelcat_map} are not equivalent.

\begin{example}\label{ex:srelcat_S1}
Let $X= S^{1} \times S^{1}$, and let $A=\{(x,x) \in X\}$
be the diagonal. We prove that $\srelcat(X, A)=1$.
Since $A$ is not a deformation retract of $X$, we have $1 \leq \srelcat(X, A)$.
We are done if we show that $\srelcat(X, A) \leq 1$. Take $z \in S^1, z\neq 1$. Let $L= (1, z)A \subset X$. Then the circle $L$ is a coset of $A $. One may consider $X$ as the image of exponential map on $\RR^2$.  Then $L$ can be taken as the image of the line parallel to the diagonal and passing through $(0, c)\in \RR^2$ for some $c\in (0,1)$. 
Furthermore,  $X-L$ is homeomorphic to an annulus containing the image of $A$ as an angular circle. So the inclusion $A\to X-L$ is a homotopy equivalence, and hence there is a strong deformation retraction $X-L\to A$. Hence $\srelcat(X, A) \leq 1$, and thus $\srelcat(X, A) = 1$.  
%Now, suppose that $V_1$ is a contractible subset in $X$. We ma consider $V_1$ is path connected. Then $X - V_1$ is homotopy equivalent to $X - \{pt\}$ which is homotopic to the wedge of two circles. Therefore, there does not exist an open subset $V_0 \subset X$ containing $A \cup (X - V_1)$ such that $V_0$ satisfies the homotopy condition in Definition \ref{d:clot}. Thus $\cat(X, A) > 1$. Observe that one can construct open subsets $V_0, V_1$ and $V_2$ of $X$ which satisfy Definition \ref{d:clot}. So $\cat(X, A)=2$.
\qed
\end{example}

\m One can ask the following question generalizing the Ganea conjecture \cite{Gan2}.

\begin{question}\label{gen_Ganea_conj}
Let $\pt$ be a point in $S^n$.
\begin{enumerate}[\rm (i)]
\item Which pairs of spaces $(X, Y)$ satisfy $$\cat(X \times S^n, Y \times \pt) = \cat(X, Y) + 1?$$
\item Which pairs of spaces $(X, Y)$ satisfy $$\srelcat(X \times S^n, Y \times \pt) = \srelcat(X, Y) + 1?$$
\end{enumerate}
\end{question}
We note that when $Y$ is a point in $X$, then Question \ref{gen_Ganea_conj} is known as the Ganea conjecture, and  counterexamples to this conjecture was first given by Iwase \cite{Iwa}.

\m In the rest of this section, we give a cohomological lower bound for $\srelcat(f)$.  The theorem below is a version of the cup-length estimate, see~\cite[Prop. 1.5]{CLOT}, we give a proof for completeness. 

Given a commutative ring $R$, the {\it nilpotency index} of $R$ is the non-negative integer $n$ such that $R^n \neq 0$ but $R^{n+1} =0$; it is denoted by $\mbox{nil} (R)$.  Let $f \colon (B, C) \to (X,A)$ be a map of pairs, and let 
\[
f^* \colon H^*(X, A) \to H^*(B, C)
\]
 be the induced homomorphism. 

\begin{thm}\label{thm:cohom_low_bnd}
${\rm nil} ({\rm Im }(f^*)) \leq \srelcat(f)$.
\end{thm}

\begin{proof}
Let $\srelcat(f) =k$ and  $g_i \colon (B_i, B_i \cap C) \to (X, A), i=0,\ldots, k$ be a restriction of $f$ that is relatively inessential, where $B_0\cup \cdots \cup B_k=B$. For the triple $(B, B_i, B_i \cap C)$, note the inclusions $q_i: B_i\cup C\to B_i$ and $\iota_i: B_i\to B$. We have the following long exact sequence
\[
\cdots \to H^*(B, B_i) \xrightarrow{q_i^*} H^*(B, B_i \cap C) \xrightarrow{\iota_i^*} H^*(B_i, B_i \cap C)  \to \cdots .
\]
 Also we have the following commutative diagrams induced from natural inclusions and the restrictions $f_i: (B, B_i\cap C)\to (X,A)$ of $f$.
\[
\begin{tikzcd}[column sep=1.5em]
 & H^*(X, A) \arrow{dr}{g_i^*} \arrow[swap]{dl}{\bar{f}^*_i} \\
H^*(B, B_i\cap C) \arrow{rr}{\iota_i^*} &&   H^*(B_i, B_i \cap C)
\end{tikzcd}
\]
and
\[
\begin{tikzcd}[column sep=1.5em]
 & H^*(X, A) \arrow{dr}{\bar{f}_i^*} \arrow[swap]{dl}{f^*} \\
H^*(B, C) \arrow{rr}{\bar{\iota}_i^*} &&   H^*(B, B_i \cap C).
\end{tikzcd}
\]
 
Suppose $\beta_0, \beta_1, \ldots, \beta_k$ belong to ${\rm Im}(f^*)$. Then $\beta_i = f^*(\alpha_i)$ for some $\alpha_i \in H^*(X, A)$ and $i \in \{0, \ldots, k\}$. Since $g_i$ is relatively inessential, $g_i^*=0$. So $\iota^*_i(\bar{f}^*_i(\alpha_i))=0 \in H^*(B_i, B_i \cap C) $.  Thus $\bar{f}^*_i(\alpha_i)= q_i^*(\gamma_i)$ for some $\gamma_i \in H^*(B, B_i)$ and $i \in \{0, \ldots, k\}$. By taking relative cup product and using $B = \cup B_i$, we get
\[
\gamma_0 \smile \gamma_1 \smile \cdots \smile \gamma_k \in H^*(B, B) =0,
\]
and $\bar{f}^*_0(\alpha_0) \smile \bar{f}^*_1(\alpha_1) \smile \ldots \smile \bar{f}^*_k(\alpha_k) \in H^*(B, C)$. Then using Property 8 in \cite[Chapter 5, Section 6]{Spa} we get
\begin{align*}
\beta_0 \smile \beta_1 \smile \cdots \smile \beta_k & =  f^*(\alpha_0) \smile f^*(\alpha_1) \smile \cdots \cup f^*(\alpha_k) \\
& =  \bar{f}_0^*(\alpha_0) \smile \bar{f}_1^*(\alpha_1) \smile \cdots \smile \bar{f}_k^*(\alpha_k) \\
& =  q_0^*(\gamma_0) \smile q_1^*(\gamma_1) \smile \cdots \smile q_k^*(\gamma_k)\\
& = q^*(\gamma_0 \smile \gamma_1 \smile \cdots \smile \gamma_k) \\
&= 0.\nonumber
\end{align*}
 This proves the conclusion.
\end{proof}

%======================================================

\section{Topological complexity of a map: several versions}\label{sec_top_comp_map}

In a different context the ``topological complexity of a map'' has been studied in  the works of \cite{Pav}, \cite{ RaDe}, and \cite{MuWu}.  In this section, we introduce the concept of the ``(higher) topological complexity'' of a map in a different way and show that it is a homotopy invariant. Then we compare it with the previous ones and study some properties of this new invariant.

\m   Consider the free path  fibration
\begin{equation}\label{pi}
\pi  \colon X^I \to X \times X, \quad \pi(\ga)=(\ga(0), \ga(1)), \, \ga \colon I\to X.
\end{equation}
 Farber~\cite{Far} defined the topological complexity $\TC(X)$ of a space $X$ as the sectional category of $\pi$, and showed a nice application of this notion to robot motion planning \cite{Far2}. Later Rudyak \cite{Rud10}, see also \cite{BGRT} introduced  the “higher analogues” of topological complexity of a space (also related to robotics, by the way). Let us recall the definition.
 
\m Given a CW complex $X$, consider the fibration
\begin{equation}\label{pin}
\begin{aligned}
\pi_n: X^{I}&\to X^n, n\geq 2\\
\pi_n(\ga)=&\left(\ga(0), \ga\left(\frac{1}{n-1}\right), \ldots, \ga\left(\frac{n-2}{n-1}\right), \ga(1)\right)
\end{aligned}
\end{equation}
where $\ga\in X^I$.
 
 \begin{df}\label{def_higher_tcn}
A {\it higher}, or {\it sequential  topological complexity} of
order $n$ of a space $X$ (denoted by $\TC_n(X)$) is the sectional category of
$\pi _n$. That is, $\TC_n(X) = \secat(\pi_n)$.
\end{df}
Note that $\TC_2$ coincides with the invariant $\TC$ introduced by Farber.

\m Let $\Delta_n \colon X\to X^n, \, \Delta_n(x)=\{(x,\cdots, x) ~\vert ~ x \in X\}$ be the diagonal map for $n \geq 2$. Consider the subspace (diagonal) $\Delta_n(X)$ of the product $X^n$.

\m Recall that the LS category of a map $f$ has two important properties: $\cat(f)$ is an invariant of the homotopy class of $f$, and $\cat(\id_X)=\cat(X)$.

Since $\cat$ is a close relative of $\TC$, it seems reasonable and useful to introduce the topological complexity of a map $f$ having, in particular, properties that are similar to those mentioned above.

\begin{definition}\label{def:TC_map}
Let $X$ and $Y$ be path connected spaces, and $f \colon X \to Y$  a map. Let $f^n := f \times \cdots \times f:X^n\to Y^n$ and
\begin{equation}\label{eq:map_tc}
\ov{f^n} \colon (X^n, \Delta_n (X)) \to (Y^n, \Delta_n(Y))
\end{equation}
be the induced maps of pairs.
 We define the $n$-th topological complexity of $f$, denoted by $\TC_n(f)$,
 as $\TC_n(f) := \qscat(\ov{f^n})$.
\end{definition}  

\m Proposition \ref{prop_inv_qscatf} (i) gives the following.
\begin{prop}
The number $\TC_n(f)$ is a homotopy invariant for $n \geq 2$.
\end{prop}

 Let $i_n \colon X^n \to X^{n+1}$ be
defined by $(x_1, \ldots, x_n) \mapsto (x_1, \ldots, x_n, x_{n})$. Then it is an embedding and $i_n(\Delta_n(X)) = \Delta_{n+1}(X)$.  
Let  $\pr_n \colon Y^{n+1} \to Y^{n} $ be  the projection on the first $n$ factors. 
\begin{prop}\label{prop:inequality_3} 
We have
\begin{itemize}
\item[(i)] $\relcat (\ov{f^n}) = \relcat(\pr_n \circ \ov{f^{n+1}} \circ i_n) \leq \relcat (\ov{f^{n+1}})$.
\item[(ii)]$\qscat (\ov{f^n}) = \qscat(\pr_n \circ \ov{f^{n+1}} \circ i_n) \leq \qscat (\ov{f^{n+1}})$.
\item[(iii)] $\srelcat (\ov{f^n}) = \srelcat(\pr_n \circ \ov{f^{n+1}} \circ i_n) \leq \srelcat (\ov{f^{n+1}})$.
\end{itemize}
\end{prop}

\begin{proof}
This follows from Proposition \ref{prop_hom_inv_relcatf} (ii), \ref{prop_inv_qscatf} (ii) and  \ref{prop_inequality_1} (ii) together with the commutativity of the following diagram.
\[
\begin{tikzcd}[column sep=1.5em]
(X^{n+1}, \Delta_{n+1}(X))  \arrow{r}{\ov{f^{n+1}}}  & (Y^{n+1}, \Delta_{n+1}(Y)) \arrow{d}{pr_n}  \\
(X^n, \Delta_n(X)) \arrow{u}{i_n} \arrow{r}{\ov{f^n}} & (Y^n, \Delta_n(Y)).  
\end{tikzcd}
\]
\end{proof}

\forget
\m From Definition \ref{d:relcat_map}, \ref{d:A-cat_map} and \ref{d:srelcat_map}, Proposition \ref{prop:ineqality1} and the note after Definition \ref{d:srelcat}, we get the following.

\begin{prop}\label{prop_ineqality_tcs}
Let $X, Y$ be CW complexes and $f \colon X \to Y$ a map. Then, for $2 \leq n \in \ZZ$, $\TC_n(f) \leq \min \{\rTC_n(f),\sTC_n(f)\}$. In particular $\TC_n(X) \leq \min \{\rTC_n(X),\sTC_n(X)$.
\end{prop}
\forgotten

\m Previously the following three version of the topological complexity of a map had been defined: the first is that given by Pave\v{s}i\'{c}, \cite{Pav}, denoted by $\TCP$, the second one is that given by Rami and Derfoufi, \cite{RaDe}, denoted by $\TCRD$, and the third is that given by Murillo and Wu \cite{MuWu}, denoted by $\TCMW$. (It  is worth noting that Pave\v{s}i\'{c} \cite{Pav} uses the non-normalized $\TC$, and so his $\TC$ is one greater then $\TCP$.) Here we have the equalities
\[
\TCP(\id_X)=\TC(X)=\TC_2(\id_X)=\TCRD(\id_X)= \TCMW(\id_X).
\]
 But neither $\TCP$ nor $\TCRD$ are homotopy invariant, see Examples \ref{example_not_homotopy_pav} and \ref{example_not_homotopy_rade}. Next we compare them with our definition of $\TC_2(f)$.

\m First we recall the definition of the topological complexity $\TCP(f)$ of a map $f$. Let $q \colon E \to B$ be a surjective map. The sectional number, denoted by $\mbox{sec}(q)$, of $q$ is the smallest length $n$ of the filtrations of open subsets $$\emptyset = V_0 \subset V_1 \subset \cdots \subset V_n=B$$ such that there is a section of $q^{-1}(V_i - V_{i-1}) \to V_i - V_{i-1}$ for $i=1, \ldots, n$. If there is no such integer then $\mbox{sec}(p)=\infty$. We note that sectional number and sectional category of $q$ are equal if $q$ is a fibration. Let $X, Y$ be path connected spaces and  $f \colon X \to Y$ a surjective map.  Consider the fibration
$
\pi  \colon X^I \to X \times X
$
as in \eqref{pi}.
It induces a continuous map $\pi_f \colon X^I \to X \times Y$ by $\pi_f = (\id \times f ) \circ \pi$. Now, the topological complexity $\TCP(f)$ of $f$ is defined as  the sectional number of $\pi_f$, that is
\begin{equation}
\label{def_tc_pav}
\TCP(f) := \mbox{sec}(\pi_f).
\end{equation}

We note that \cite[Corollary 3.9]{Pav} says that $\TCP$ is a fiber homotopy equivalence invariant  of a map over the base.

\begin{example}\label{example_not_homotopy_pav}
Let $f, g \colon [0, 3] \to [0,2]$ be two continuous functions defined by the following.
\begin{align*}
f (x) =
     \begin{cases}
       x \quad&\text{if}~ 0 \leq x \leq 1 \\
       1  \quad&\text{if}~1 \leq x \leq 2 \\
       x-1  \quad&\text{if}~ 2 \leq x \leq 3\\
     \end{cases}
\end{align*}
and
\begin{align*}
g (x) = \frac{2x}{3}  \quad\quad \text{if}~ 0 \leq x \leq 3.
     \end{align*}
     
\m Hence $f$ has no continuous section, and $g$ has a section. So, the first doted  remark in \cite[Page 111]{Pav} gives   $\TCP(f) > 1$, and by \cite[Proposition 3.3]{Pav} $\TCP(g)=1$. But $f$ and $g$ are homotopic by the linear combination $t f +(1-t)g, t \in I$.
Therefore $\TCP$ is not a homotopy invariant of a map. This was mentioned in \cite{Pav}, but an explicit example was not given.
\end{example}

\m Now we recall the topological complexity $\TCRD(g)$ of a map  $g$. Let $Z$ be a path connected space and $g \colon Z \to W$ be a map. Then the space $Z \times_{W} Z :  = (g \times g)^{-1}(\Delta_2 W)$ is a subset of $Z \times Z$. Let
\[
\pi^g \colon \pi^{-1}(Z \times_{W} Z) \to Z \times_W Z
\]
be the pullback induced from the fibration $\pi \colon Z^I \to Z\ts Z$  by the inclusion

\[
Z \times_{W} Z \to Z\ts Z.
\]

Then $\TCRD(g)$ is the sectional category of $\pi^g$, that is

\begin{equation}\label{def_tc_rade}
\TCRD(g) := \sec(\pi^g).
\end{equation}
 
\m It turns out that $\TCRD$ is also a fiber homotopy equivalence invariant \cite[Corollary 7]{RaDe}. We also  note that this definition is a particular case of ``relative topological complexity" studied in \cite{Far2}.

\begin{example}\label{example_not_homotopy_rade}
In this example we show that $\TCRD$, topological complexity of a map defined in \cite{RaDe}, is not a homotopy invariant of a map. Let $X$ be a path connected topological space and $CX$ the cone on $X$ with apex $a$. Write $CX=X \ts I/X\ts\{1\}$. Define $\iota, \mathfrak{c} \colon X \to CX$ where $\iota$ is the inclusion $\iota(x)=(x,0)$ and  $\mathfrak{c}(X)=a$. Then  $\TCRD(\mathfrak{c}) = \TC(X)$ (Farber's topological complexity).

For the map $\iota$, we have $X \times_{CX} X = \Delta_2(X)$. Then $\pi^{-1}(\Delta _2(X))$ is the free loop space on $X$. So constant maps induce a section of $\pi^{\iota}$. Therefore $\TCRD(\iota) =1 \neq \TC(X)=\TCRD (\mathfrak c)$ in general.  It remains to note that $\iota$ and $\mathfrak{c}$ are homotopic.
\end{example}

\m In contrast to both definitions $\TCP$ and $\TCRD$ of the ``topological complexity of a map", for any $n \geq 2$, the number $ \TC_n(f)$ is a homotopy invariant, a topologist's primary interest.

Now we show explicitly that $\TCP(f) \neq \TC_2(f) \neq \TCRD(f)$ for some maps $f$. Let $X$ be a contractible space and $Y$ a non-contractible space. Let $f \colon X  \to Y$ be a surjective fibration, for example $\exp \colon I \to S^1$. Then by \cite[Proposition 3.2]{Pav} ${\TCP(f)} \geq \cat(Y) +1 \geq 2$.  But our definition gives $\TC_2(f) = 0$. Therefore these two are different.

\vspace{0.1 cm}   On the other hand, if $g \colon X \to Y$ is a constant map then $\TCRD(g)$ is the topological complexity $\TC(X)$ of $X$ which is strictly greater than $\cat(X)$ in general. Note that in this case $\TC_2(g)=0$.

\m Now we recall the definition of $\TCMW(f)$ of a (work) map.  Let $f \colon X \to Y$ be a continuous map. Then $\TCMW(f)$ is the least integer $n \leq \infty$ such that there exist open subsets $U_0, \ldots, U_n$ of $X \times X$ on each of which there is a map $s_i \colon U_i \to X$ satisfying $(f\times f) \circ \Delta_2 \circ s_i \simeq (f \times f)|_{U_i}$ for $i=0, \ldots, n$. We also note that Scott  defines topological complexity $\TCS(f)$ of a map in \cite[Definition 3.1]{Sco}, and he shows that these two definitions are equivalent.

\m Let $(a,b) \in U_i \subseteq X\times X$. So $s_i(a,b) \in X$. Then $(f\times f) \circ \Delta_2 \circ s_i (a, b) \in \Delta_2(Y)$.  Yet, we cannot conclude that the restriction
\[
(f \times f)|_{U_i} \colon (U_i, U_i \cap \Delta_2(X))\to (Y, \Delta_2(Y))
\]
 is inessential in the sense of Definition \ref{def_rel_inessential}.
 However, if we assume  $(f \times f)|_{U_i}$ is inessential, then the map $f \ts f: U_i\to Y\ts Y$ is homotopic to a map $U_i \to \Delta_2(Y)$. Thus by \cite[Theorem 3.4]{Sco} we get the following.  
 \begin{equation}\label{TC_RS_MW}
 \TCMW(f) = \TCS(f) \leq  \TC_2(f).
 \end{equation}
 We do not yet know if $\TC_2(f)=\TCMW(f)$.

\m In the remaining we study some properties and upper bound of $\TC_n(f)$.
The following result generalizes some parts of \cite[Proposition 3.3]{Rud10}, Proposition 3.1 and Corollary 3.3 in \cite{BGRT}.

\begin{proposition}\label{prop_cat_tc}
Let  $f \colon X \to Y$ be a map between CW spaces. Then
\[
\TC_n(f) \leq \TC_{n+1}(f) \leq \cat(f^{n+1}) \leq (n+1) \cat(f)
\]
 for all $n\geq 2$. In particular, the growth of $\TC_n(f)$ is linear with respect to $n$.
\end{proposition}

\begin{proof} The first inequality follows from Definition \ref{def:TC_map} and Proposition \ref{prop:inequality_3}(ii). The second inequality follows from the definitions of $\TC_n(f)$ and $\cat(f^n)$. The last inequality follows from \eqref{eq_prod}.
\end{proof}
 
Using the arguments in the proof of \cite[Lemma 3.5]{LuMa}, one can show the following.
\begin{proposition}\label{prop:tcnx}
If $\id \colon X \to X$ is the identity map, then $\TC_n(X) = \qscat(\ov{\id^n})$.
\end{proposition}
\forget
Therefore, one can see that $\TC_n$ turns out to be functor on the category of
topological spaces.
\forgotten

\begin{proposition}\label{prop_tcxy}
Let  $f \colon X \to Y$ be a map  between connected {\rm CW}-complexes. Then $\TC_n(f) \leq \min \{\TC_n (X), \TC_n (Y) \}$.
\end{proposition}
\begin{proof}
This follows from Proposition \ref{prop_inv_qscatf} (ii) and \ref{prop:tcnx}.  
\end{proof}
It is well known that $\TC_n (X) \leq \cat(X^n)$ for any path-connected space. So $$\TC_n(f) \leq \min \{n \cat(X), n \cat(Y)\}.$$

 \begin{example}
 Let $f \colon S^{1} \to S^1$ be the map defined by $z \mapsto z^p$ where $p > 0$.
 So $\deg(f)=p$, and so $f$ is essential. Thus $1 \leq \TC_2(f)$. On the other hand $\TC_2(f) \leq \srelcat((S^1)^2, \Delta_2(S^1))$. So, by Example \ref{ex:srelcat_S1}, $\TC_2(f)=1$ and it does not depend on the degree of $f$.
\end{example}
 
 \forget
\begin{example}
Let $\mathbb{RP}^3 \xrightarrow{g} \mathbb{RP}^3/\mathbb{RP}^{2} \cong S^3$ be the natural quotient map. So $g$ is essential which implies that  $1 \leq \TC_2(g) \leq \sTC_2(g)$. On the other hand, $\sTC_2(g) \leq \srelcat((S^3)^2, \Delta_2(S^3)) = 1$ by Example \ref{ex:srelcat_S3}. Thus $\TC_2(g)=1 = \sTC_2(f)$, whereas the results in \cite{MuWu} cannot determine $\TCMW(g)$. However, we have
 \[
 1 \leq \TCMW(g) \leq \TC_2(g)=1.
 \]
 \end{example}
 \forgotten

\section{Weak topological complexity of a map}\label{sec_wtop_comp_map}
 
In this section we recall the concept of weak sectional category and weak topological complexity following \cite{GaVa}. Then we introduce the concept of weak higher topological complexity of a map and study some of its properties.

Let $X$ be a space and $k \geq 1$. Recall that  $\Delta_k \colon X \to X^k$ is the diagonal map. Let $X^{[k+1]}$ be  the $(k + 1)$-fold smash-product. Then the weak category of $X$, denoted by $\wcat(X)$, is the least integer $k$ such that the composition  $X \xrightarrow{\Delta_{k+1}} X^{k+1} \xrightarrow{q} X^{[k+1]}$ is null homotopic where $q$ is the quotient map. The paper \cite{BeHi} introduced this and showed that $\wcat(X) \leq \cat(X)$. 

Let $ \xi \colon E \to X$ be a map. Then the $k$th fatwedge of $\xi$ for $0 \leq k \in \ZZ$ is the map $$\xi_k \colon T^k(\xi) \to X^{k+1}$$ which is defined inductively by the following.
\begin{enumerate}
\item $\xi_0 := \xi \colon E \to X$ and $T^0(\xi) := E$.
\item Assume that $\xi_{k-1} \colon T^{k-1}(\xi) \to X^k$ is defined. Then $\xi_k$ and $T^k(\xi)$ is constructed from the following homotopy commutative diagram.
\begin{center}
\begin{tikzcd}
\bullet \arrow[rr] \arrow[dd]  && E \times X^k \arrow[dl] \arrow[dd,"\xi \times {\id}_{B^k}"]  \\
 & T^k(\xi)\arrow[dr, dotted, "\xi_k"]  &\\
X \times T^{k-1}(\xi) \arrow[rr,"{\id}_B \times \xi_{k-1}"] \arrow[ur]  && X^{k+1}.
\end{tikzcd}
\end{center}
That is, $\xi_k$ is the join of the maps ${\id}_B \times \xi_{k-1}$ and $\xi \times {\id}_{B^k}$.
\end{enumerate}

\begin{df}\cite[Definition 1.4]{GaVa}
Let $\xi \colon E \to X$ be a map and $C_k(\xi)$ the homotopy cofiber of $\xi_k \colon T^{k}(\xi) \to X^{k+1}$ with an induced map $ \xi_{M_k} \colon X^{k+1} \to C_k(\xi)$. Then the weak sectional category of $\xi$, denoted by $\wsecat(\xi)$, is the least positive integer $k$ such that the composition $X \xrightarrow{\Delta_{k+1}} X^{k+1} \xrightarrow{\xi_{M_k}} C_k(\xi)$ is inessential.  
\end{df}

In particular, one can define the following.
\begin{df}
The $n$th weak topological complexity of $X$, denoted by $\wTC_n(X)$, is the weak sectional category of $\pi_n$ defined in \eqref{pin}. That is $\wTC_n(X) : = \wsecat(\pi_n)$ for $n \geq 2$.  
\end{df}

Note that when $n=2$ then $\wTC_2(X)$ is $\wTC(X)$ of \cite[Section 3]{GaVa}. Also $X$ is homotopy equivalent to $X^I$ via the map $x \mapsto c_x$ where $c_x \colon I \to X$ is the constant map defined by $c_x(t) =x$. That is $\pi_n$ is the fibrational substitution of the  diagonal map $\Delta_n$ for $n \geq 2$. Therefore,  $\wTC_n(X) = \wsecat(\Delta_n)$. 
Now we extend this concept to define weak topological complexity of a map.

\begin{df}\label{def_wsecat_f}
Let $f \colon X \to Y$ be a map. Then the weak topological complexity of $f$, denoted by $\wTC_n(f)$, is the weak sectional category of the composition $X \xrightarrow{\Delta_n} X^{n} \xrightarrow{f^{n}} Y^{n}$.
\end{df}
In particular, if  $\id \colon X \to X $ is the identity map then $\wTC_n({\id})= \wTC_n(X)$ for $n \geq 2$.

Let $R$ be a commutative ring with unity. Then, for $n \geq 2$, the element $u\in H^*(X^n;R)$ is called a zero-divisor class of grade $n$ if $\Delta_n^*u=0$ where $\Delta_n \colon X \to X^n$ is the diagonal map. 
 The zero-divisors-cup-length of grade $n$, denoted by $\zcl_n^R(X)$, for $X$ is the maximal $k$ such that $u_1\smile \cdots \smile u_k\neq 0$ provided each $u_i$ is a zero-divisor class of grade $n$.
\begin{proposition}
Let $f \colon X \to Y$ be a map and $g_n = f^n \circ \Delta_n$ for $n \geq 2$. Then,
 $\wTC_n(f) \leq \wcat(Y^n)$ and $\wTC_n(f) \geq \emph{nil}(\emph{ker} g_n^*)$. 

In particular,  $\wTC_n(X) \leq \wcat(X^n)$ and $\wTC_n(X) \geq \zcl_n^R(X)$. 
\end{proposition}
\begin{proof}
This follows from \cite[Theorem 21]{GaVa} and Definition \ref{def_wsecat_f}. 
\end{proof}

{\bf Acknowledgment:} The second author thanks University of Florida for supporting his visit, `International office IIT Madras' and `Science and Engineering Research Board India' for research grants. The authors thank John Oprea, Jamie Scott, Petar Pavesic, and Alex Dranishnikov for some helpful discussion.  

%%%%%%%%%%%%%%%%%%%%%%%%%%%%%%%%%%%%%%%%%%%%%%%%%%%%

\renewcommand{\refname}{References}

\end{document}